\documentclass[12pt]{article}
\usepackage{amsmath,amssymb,amsthm,enumitem}

\newtheorem{theorem}{Theorem}

\newcommand\E{\mathbb{E}}
\newcommand\Prb{\mathbb{P}}
\newcommand\cB{\mathcal{B}}
\newcommand\cE{\mathcal{E}}
\newcommand\Bin{\mathrm{Bin}}

\advance\textwidth2\oddsidemargin
\title{The asymptotic number of prefix normal words}
\author{Paul Balister\thanks{Department of Mathematical Sciences, University of Memphis, Memphis TN 38152.
Email: \texttt{pbalistr@memphis.edu}. Partially supported by NSF grant DMS 1600742.}
\and Stefanie Gerke\thanks{Mathematics Department, Royal Holloway University of London, Egham TW20\thinspace0EX, UK.
Email: \texttt{Stefanie.Gerke@rhul.ac.uk}.}}

\begin{document}

\maketitle

\begin{abstract}
We show that the number of prefix normal binary words of length $n$ is $2^{n-\Theta((\log n)^2)}$.
We also show that the maximum number of binary words of length $n$ with a given fixed prefix normal form
is $2^{n-O(\sqrt{n\log n})}$.
\end{abstract}


\section{Introduction}

Given a binary word $w=(w_i)_{i=1}^n\in\{0,1\}^n$ of length~$n$, denote by $w[j,k]$ the subword of length $k-j+1$
starting at position $j$ and ending at position~$k$, that is, $w[j,k]=w_jw_{j+1}\dots w_k$.
Let $|w|_1$ be the number of 1s in the word~$w$. We define the \emph{profile\/}
$f_w\colon\{0,\dots,n\}\to\{0,\dots,n\}$ of $w$ by
\[
 f_w(k)=\max_{0\le j\le n-k}|w[j+1,j+k]|_1,
\]
so that $f_w(k)$ is the maximum number of 1s in any subword of $w$ of length~$k$.
The word $w$ is called \emph{prefix normal\/} if for all $0\le k\le n$ this number is maximized at $j=0$,
so that
\[
 |w[1,k]|_1\ge|w[j+1,j+k]|_1\qquad\text{for }0\le j\le n-k.
\]
In other words, a word $w$ is called prefix normal if the number of $1$s in any subword
is at most  the number of $1$s in the prefix of the same length.

If $j<k$ then we can remove the common subword $w[j+1,k]$ of $w[1,k]$ and $w[j+1,j+k]$,
so that $|w[1,k]|_1\ge|w[j+1,k+j]|_1$ iff $|w[1,j]|_1\ge|w[k+1,k+j]|_1$.
Thus to show that $w$ is prefix normal it is enough to check that
\begin{equation}\label{e1}
 |w[1,k]|_1\ge |w[j+1,j+k]|_1\qquad\text{for }k\le j\le n-k.
\end{equation}

Prefix normal words were introduced by G.~Fici and Z.~Lipt\'ak in \cite{pnf2conf} because
of their connection to binary jumbled pattern matching. Recently, prefix normal words
have been used because of their connection to trees with a prescribed number of vertices
and leaves in caterpillar graphs~\cite{caterpillar}.

The number of prefix normal words of length $n$ is listed as sequence A194850 in
The On-Line Encyclopedia of Integer Sequences (OEIS)~\cite{OEIS}. We prove the following result,
conjectured in~\cite{pnf1} (Conjecture 2) where also weaker
upper and lower bounds were shown, see also~\cite{pnf1x}.

\begin{theorem}\label{t:1}
 The number of prefix normal words of length $n$ is $2^{n-\Theta((\log n)^2)}$.
\end{theorem}

Given an arbitrary binary word $w$ of length $n$, the \emph{prefix normal form}
$\tilde w$ of $w$ is the unique binary word of length $n$ that satisfies
\[
 |\tilde w[1,k]|_1=f_w(k).
\]
Note that for any $w$, $f_w(k)\le f_w(k+1)\le f_w(k)+1$, so $\tilde w$ is well-defined.
Moreover, we can define an equivalence relation $\sim$ on binary words of length $n$
by
\[
 w\sim v \qquad\Longleftrightarrow\qquad f_w=f_v \qquad\Longleftrightarrow\qquad \tilde w=\tilde v.
\]
Indeed, $\tilde w$ is just the lexicographically maximal element of the equivalence class
$[w]$ of $w$ under this equivalence relation.

In~\cite{pnf2conf} it is asked how large can an equivalence class $[w]$ be. In other words,
what is the maximum number of words of length $n$ that have the same fixed prefix normal form.
This maximum number is listed in the OEIS as sequence A238110~\cite{OEIS}.
From Theorem~\ref{t:1} it is clear that it must be at least $2^{\Theta((\log n)^2)}$.
However, we show that it is much larger.

\begin{theorem}\label{t:2}
 For each $n$ there exists a prefix normal word $w$ such that the number of binary
 words of length $n$ with prefix normal form $w$ is $2^{n-O(\sqrt{n\log n})}$.
\end{theorem}

\section{Proofs}

\begin{proof}[Proof of the lower bound of Theorem~\ref{t:1}.]
To prove the lower bound we will need to construct $2^{n-\Theta((\log n)^2)}$ prefix normal
words of length~$n$. We will do so by giving a random construction and showing that this
construction almost always produces a prefix normal word.

Fix a constant $c> \sqrt{2}$ and define
\[
 p_k=\begin{cases}
 \frac{1}{2}+c\sqrt{\frac{\log n}{k}},&\text{for }k>16c^2\log n;\\
 1,&\text{for }k\le 16c^2\log n.
\end{cases}
\]
Write $k_0:=\lfloor 16c^2\log n\rfloor$ so $p_k=1$ if $k\le k_0$, and $p_k\in[\frac12,\frac34]$
for $k>k_0$. Let $w$ be a random word with each letter $w_k$ chosen to be 1 with probability~$p_k$,
independently for each $k=1,\dots,n$.
Clearly \eqref{e1} holds for all $k\le k_0$, so assume $k>k_0$.
By comparing the integral $\int c\sqrt{\frac{\log n}{k}}\,dk=2c\sqrt{k\log n}+C$
with the corresponding Riemann sum, we note that
\[
 \sum_{i=1}^k p_i = \tfrac{k}{2}+2c\sqrt{k\log n} + O(1)
\]
uniformly for $k>k_0$ (and uniformly in~$c$). Indeed, the approximation of the integral by the Riemann sum
has error at most the maximum term, due to the monotonicity of the integrand, and the additive constant is also $O(1)$ by considering the case $k=k_0$.
From this we estimate the expected difference
\begin{equation}\label{difference}
 |w[1,k]|_1-|w[j+1,j+k]|_1=\sum_{i=1}^kw_i+\sum_{i=j+1}^{k+j}(1-w_i)-k
 \end{equation}
as
\[
 \mu:=\E\big(|w[1,k]|_1-|w[j+1,j+k]|_1\big)
 =2c\sqrt{k\log n}-2c\sqrt{(j+k)\log n}+2c\sqrt{j\log n}+O(1).
\]
This expression is minimized when $j$ is as small as possible, i.e., $j=k$. Thus
\[
 \mu\ge 2(2-\sqrt2)c\sqrt{k\log n}+O(1)> c\sqrt{k\log n}
\]
for sufficiently large $n$.
By~\eqref{difference},  $|w[1,k]|_1-|w[j+1,j+k]|_1$ can be considered as the sum of
$2k$ independent Bernoulli random variables (with an offset of $-k$).

We recall the \emph{Hoeffding bound\/}~\cite{hoeffding} that states
that if $X$ is the sum of $n$ independent random variables in the interval $[0,1]$
then for all $x\ge 0$,
\begin{equation}\label{eq:hoeffding}
 \Prb\big(X-\E(X)\ge x\big)\le \exp\{-2x^2/n\}\quad\text{and}\quad
 \Prb\big(X-\E(X)\le -x\big)\le \exp\{-2x^2/n\}.
\end{equation}
(Note that these two bounds are essentially the same bound as the second can be easily
derived from the first by exchanging the roles of the $0$s and $1$s but we state
them both here for convenience.)

Let $\mu^*= \E\big(\sum_{i=1}^k w_i +\sum_{i=j+1}^{k+j}(1-w_i)\big)$. Note that $\mu^*=\mu+k$.
We have
\begin{align*}
 \Prb\big(|w[1,k]|_1<|w[j+1,j+k]|_1\big)
 &\stackrel{\eqref{difference}}{=} \Prb\left(\sum_{i=1}^k w_i +\sum_{i=j+1}^{k+j}(1-w_i) <k \right)\\
 &\le \Prb\left(\sum_{i=1}^k w_i +\sum_{i=j+1}^{k+j}(1-w_i) -\mu^* <k- \mu^* \right)  \\
 &\le \Prb\left(\sum_{i=1}^k w_i +\sum_{i=j+1}^{k+j}(1-w_i) -\mu^* <  -  \mu \right)  \\
 &\stackrel{\eqref{eq:hoeffding}}{\le} \exp\big\{-2\mu^2/(2k)\big\}\\
 &\le\exp\big\{-c^2\log n \big\}
\end{align*}
Hence if $c$ is large enough ($c>\sqrt{2}$) then $\Prb(|w[1,k]|_1<|w[j+1,j+k]|_1)=o(n^{-2})$.
Taking a union bound over all possible values of $k$ and~$j$,
we deduce that $w$ is prefix normal with probability $1-o(1)$.

It remains to count the number of such~$w$. For any discrete random variable $X$, define the
\emph{entropy\/} of the distribution of $X$ as
\[
 H(X):=\sum_x -\Prb(X=x)\log_2\Prb(X=x),
\]
where the sum is over all possible values $x$ of $X$ and the logarithm is to base~2.
If the random variable is a Bernoulli random variable, 
we call $H(\mathrm{Be}(p))$ the {\it binary entropy function}~$H_b(p)$.
 We use the following
well-known (and easily verified) facts about the entropy.
\begin{enumerate}[label=H\arabic*)]
\item  If $X_1,\dots,X_n$ are independent discrete random variables and $X=(X_1,\dots,X_n)$,
then $H(X)=\sum_{i=1}^n H(X_i)$.
\item If $X$ takes on at most $N$ possible values with positive probability then $H(X)\le \log_2 N$.
\item \label{H3} The Taylor series of  the binary entropy function in a neighbourhood of $1/2$ is
\[ H_b(p)=1 - \frac{1}{2\ln 2}\sum_{n=1}^\infty \frac{(1-2p)^{2n}}{n(2n-1)}.\]
In particular, for a Bernoulli random variable with $\Prb(X=1)=\frac12+x$, $H(X)=1-\Theta(x^2)$.
\item If $\cB$ is subset of possible values of $X$ we have
\[
 H(X)=H(X\mid X\in\cB)\Prb(X\in\cB)+H(X\mid X\notin\cB)\Prb(X\notin\cB)+H(1_{X\in\cB}),
\]
where $X\mid \cE$ denotes the distribution of $X$ conditioned on the event $\cE$
and $1_{\cE}$ denotes the indicator function of~$\cE$.
\end{enumerate}

Applying these results to our random word $w$ we have
\[
 H(w)=\sum_{k>k_0}^n H(w_k)=n-k_0-\Theta\left(\sum_{k=k_0}^n c^2\tfrac{\log n}{k}\right)
 =n-\Theta((\log n)^2).
\]
On the other hand, if $\cB$ is the set of prefix normal words, then
\begin{align*}
 H(w)&=H(w\mid w\in\cB)\Prb(w\in\cB)+H(w\mid w\notin\cB)\Prb(w\notin\cB)+H(1_{w\in\cB})\\
 &\le \log_2(|\cB|)\Prb(w\in\cB)+n\,\Prb(w\notin\cB)+1\\
 &= n+1 - (n-\log_2|\cB|)(1-o(1)).
\end{align*}
We deduce that $n-\log_2|\cB|\le\Theta((\log n)^2)$ and hence $|\cB|\ge 2^{n-\Theta((\log n)^2)}$.
\end{proof}

\begin{proof}[Proof of the upper bound in Theorem~\ref{t:1}.]
We will prove the upper bound in two parts. Firstly we will show that most prefix normal words
have to contain a good number of $1$s in any prefix of reasonable size 
as we cannot extend a prefix with too few 1s to a prefix normal word  in many ways.
Secondly, we will show that there are at most $2^{n-\Theta (\log^2n)}$ ways to construct a word
which has sufficiently many $1$s in all reasonably sized prefixes.

Assume $\log n\le k\le \sqrt n$ and consider the first $\lfloor\sqrt n\rfloor$ blocks
of size $k$ of~$w$. If $|w[1,k]|_1=d$ then the number of choices for the second and subsequent blocks
is at most $2^k(1-\Prb(\Bin(k,\tfrac12)>d))$, and hence the number of choices for $w$ is at most
\[
 2^n\big(1-\Prb\big(\Bin(k,\tfrac12)>d\big)\big)^{\lfloor\sqrt n\rfloor-1}
 \le 2^{n-\Omega(\sqrt n\,\Prb(\Bin(k,1/2)>d))}.
\]
If $\Prb(\Bin(k,\tfrac12)>d)>n^{-1/3}$, say, then there are far fewer than
$2^{n-\Theta((\log n)^2)}$ choices of such prefix normal words, even allowing for summation
over all such $k$ and~$d$.

Using Stirling's formula one can show that  for $1/2<\lambda<1$ and $\lambda k$ integral,
\[  \Prb\big(\Bin(k,\tfrac12)\geq \lambda k) = \sum_{i=\lambda k}^k\binom{k}{i}2^{-k}\geq \frac{2^{k H_b(\lambda)-k}}{\sqrt{8k\lambda(1-\lambda)}} \geq \frac{2^{k H_b(\lambda)-k}}{\sqrt{2k}} , \]
see for example \cite{ash} for a detailed proof.

Thus, by \ref{H3}, we have 
\[
 \Prb\big(\Bin(k,\tfrac12)>\tfrac{k}{2}+x\big)\ge \frac{1}{\sqrt{2k}}2^{-\Theta(x^2/k)},
\]
provided $x<k/2$. Thus if $\log n\le k\le \sqrt n$ and $\Prb(\Bin(k,\tfrac12)>d)>n^{-1/3}$
we can deduce that $d\ge \frac{k}{2}+c\sqrt{k\log n}$ for some small universal constant $c>0$.
Thus, without loss of generality, we can restrict to prefix normal words with the property
that
\begin{equation}\label{e2}
 |w[1,k]|_1\ge \tfrac{k}{2}+c\sqrt{k\log n}\qquad\text{for all $k$ with}\qquad \log n\le k\le\sqrt n.
\end{equation}
Define $d_0=c\sqrt{\log n}$, which for simplicity we shall assume is an integer.
(One can reduce $c$ slightly to ensure this is the case.)
Define $\cE_t$ to be the event that \eqref{e2} holds with $k=4^t$,
i.e., that $|w[1,4^t]|_1\ge 2^{2t-1}+2^td_0$. Let $t_0$ be the smallest $t$ such that $4^t\ge\log n$
and let $t_1$ be the largest $t$ such that $4^t\le\sqrt{n}$. We bound the probability
that a uniformly chosen $w\in\{0,1\}^n$ satisfies $\cE_{t_0}\cap\cE_{t_0+1}\cap\dots\cap\cE_{t_1}$.

Write $\cE_{t,j}$ for the event that $|w[1,4^t]|_1=2^{2t-1}+2^td_0+j$ and
$\cE_{t,\ge j}$ for the event that $|w[1,4^t]|_1\ge 2^{2t-1}+2^td_0+j$. Thus $\cE_t$ is just
$\cE_{t,\ge0}$. Write $\cE_{\le t}$ for the intersection $\cE_{t_0}\cap\cE_{t_0+1}\cap\dots\cap \cE_t$.

{\bf Claim:} For $t\in[t_0,t_1]$ and $j\ge0$,
\[
 \Prb\big(\cE_{\le t-1}\cap\cE_{t,\ge j}\big)\le  n^{-2c^2(t-t_0+1)/3}\beta_t^j/(1-\beta_t),
\]
where $\beta_t:=\exp\{-2^{3-t}d_0/3\}$.
Note that $\beta_t<1$ for all $t\in[t_0,t_1]$. For the case $t=t_0$ we simply use the
Hoeffding bound \eqref{eq:hoeffding} to obtain
\begin{align*}
 \Prb(\cE_{t_0,\ge j})&=\Prb\big(\Bin(4^{t_0},\tfrac12)\ge 2^{2t_0-1}+2^{t_0}d_0+j\big)
 \le  \exp\big\{-2(2^{t_0}d_0+j)^2/4^{t_0}\big\}\\
 &\le   \exp\big\{-2d_0^2-4jd_0/2^{t_0}\big\}
 = n^{-2c^2}\beta_{t_0}^{3j/2}<n^{-2c^2/3}\beta_{t_0}^j/(1-\beta_{t_0})
\end{align*}
as required.

Now assume the claim is true for~$t$. We first want to give a bound on
$\Prb(\cE_{\le t}\cap\cE_{t+1,\ge j})$. Note that if $\cE_{\le t-1} \cap \cE_{t,i}$ holds
then in particular $\cE_{t,i}$ holds and thus for $\cE_{t+1,\ge j}$ to hold we still need at least
\[
 2^{2(t+1)-1} + 2^{t+1}d_0 +j - 2^{2t-1}-2^td_0-i = 3\cdot 2^{2t-1} + 2^{t}d_0 +j -i
\]
$1$s in the interval $[4^t +1, 4^{t+1}]$.
Thus we get
\[
 \Prb\big(\cE_{\le t}\cap\cE_{t+1,\ge j}\big)
 \le\sum_{i\ge0}\Prb(\cE_{\le t-1}\cap\cE_{t,i})
 \Prb\big(|w[4^t+1,4^{t+1}]|_1\ge 3\cdot 2^{2t-1}+2^t d_0+j-i\big).
\]

Note  that there are $4^{t+1}-4^t=3\cdot 4^t$ elements in
the interval $[4^t +1, 4^{t+1}]$ and that we expect
\[
 \frac{3\cdot 4^t}{2} = 3\cdot 2^{2t-1}
\]
$1$s in this interval.
Hence by Hoeffding
\begin{align*}
 \Prb\big(|w[4^t+1,4^{t+1}]|_1\ge 3\cdot 2^{2t-1}+2^t d_0+j\big)
 &\le \exp\big\{ -2(2^td_0+j)^2/(3\cdot 4^t)\big\}\\
 &\le \exp\big\{-2d_0^2/3-4jd_0/(3\cdot 2^t)\big\}\\
 &= n^{-2c^2/3}\beta_{t+1}^j.
\end{align*}
Note that the final inequality is even true for negative $j$:
for $j\ge - 2^td_0$ Hoeffding's bound holds, and for $j\le -2^t d_0$
the bound on the probability is larger than~$1$.
If we let $p_i=\Prb(\cE_{\le t-1}\cap\cE_{t,\ge i})$ then we have
\begin{align*}
 \Prb(\cE_{\le t}\cap\cE_{t+1,\ge j})
 &\le \sum_{i\ge 0}(p_i-p_{i+1}) n^{-2c^2/3}\beta_{t+1}^{j-i}\\
 &\le  n^{-2c^2/3}\beta_{t+1}^j\big(p_0+(1-\beta_{t+1})(\beta_{t+1}^{-1}p_1+\beta_{t+1}^{-2}p_2+\dots)\big).
\end{align*}
Now by induction, $p_i\le n^{-2c^2(t-t_0+1)/3}\beta_t^i/(1-\beta_t)$. As $\beta_t=\beta_{t+1}^2$ we have
\begin{align*}
 \Prb(\cE_{\le t}\cap\cE_{t+1,\ge j})
 &\le n^{-2c^2(t-t_0+2)/3}\beta_{t+1}^j (1+(1-\beta_{t+1})(\beta_{t+1}+\beta_{t+1}^2+\dots))/(1-\beta_{t+1}^2)\\
 &=n^{-2c^2(t-t_0+2)/3}\beta_{t+1}^j(1+\beta_{t+1})/(1-\beta_{t+1}^2)\\
 &= n^{-2c^2(t-t_0+2)/3}\beta_{t+1}^j/(1-\beta_{t+1}),
\end{align*}
as required. Thus the claim is proved.

Now we take $t=t_1$ and $j=0$ to deduce that $\Prb(\cE_{\le t_1})\le  n^{-2c^2(t_1-t_0+1)/3}/(1-\beta_{t_1})$.
Recall $\beta_{t_1} = \exp(-2^{3-t_1}d_0/3)$, $d_0=c\sqrt{\log n}$,
and that $t_1$ was chosen so $\sqrt{n}/4 < 4^{t_1} \le \sqrt{n}$.
Thus, for large~$n$,  $n^{-1/4} < 2^{3-t_1}d_0/3 < 1$.
Using the inequality $e^{-x}\le 1-x/2$, which holds for $0\le x\le 1$, we deduce that
$1-\beta_{t_1}\ge n^{-1/4}/2$, and so $1/(1-\beta_{t_1}) = O(n^{1/4})$.
Also, we have  $t_1-t_0+1=\Theta(\log n)$ as $n\to\infty$   and thus
$\Prb(\cE_{\le t_1})\le 2^{-\Omega((\log n)^2)}$.
As the probability that a uniformly chosen word $w$ satisfies $\cE_{\le t_1}$ is at most $2^{-\Omega((\log n)^2)}$,
we deduce that the number of prefix normal words is at most $2^{n-\Theta((\log n)^2)}$.
\end{proof}

\begin{proof}[Proof of Theorem~\ref{t:2}.]
Fix an integer $t\approx\sqrt{n\log n}$ and assume for simplicity that $n$ is a multiple of $2t$.
Define $w=(10)^t1^{2t}c_1c_2\dots c_{(n-4t)/2t}$, where $c_i$ are arbitrary Catalan sequences of length $2t$.
Here a \emph{Catalan sequence} is a binary sequence $c$ of length $2t$ such that $|c[1,i]|_1\le i/2$ for all $i=1,\dots,2t$
and $|c|_1=t$. It is well-known that the number of choices for $c_i$ is the \emph{Catalan number}
\[
 C_t=\frac{1}{t+1}\binom{2t}{t}\sim \frac{2^{2t}}{\sqrt\pi t^{3/2}}.
\]
It is easy to see that the prefix normal form of any $w$ of this form is
\begin{equation}\label{e3}
 \tilde w=1^{2t}(01)^{(n-2t)/2}.
\end{equation}
Indeed, there is a subword $1^k$ of $w$ for all $k\le 2t$. For $k>2t$, if we write $k=2tq+r$ with $0\le r<2t$
then we have a subword $(10)^{r/2}1^{2t}c_1\dots c_{q-1}$ or $0(10)^{(r-1)/2}1^{2t}c_1\dots c_{q-1}$
which is of length $t$ and has the requisite number $t+\lfloor k/2\rfloor$ of 1s. On the other hand,
the definition of a Catalan sequence implies no other subword of length $k$ containing the $1^{2t}$ subword
can possibly have more 1s. Any substring intersecting the $1^{2t}$ and of length greater than $2t$
can be replaced by one containing the $1^{2t}$ with at least as many ones.
And finally, any subword of $w$ length $k>2t$ not intersecting the $1^{2t}$ subword
(so contained within the $c_1\dots c_{(n-4t)/2t}$ subword) can have at most $t+\lfloor k/2\rfloor$ 1s as an end-word
of $c_i$ contains at most $t$ 1s and there are at most $\lfloor k/2\rfloor$ 1s in the initial subword
of $c_{i+1}c_{i+2}\dots$ of length~$k$.

It remains to count the number of possible $w$'s. This is just
\[
 C_t^{(n-4t)/(2t)}=2^{n-4t-(\log t)3n/4t+O(n/t)}.
\]
Taking $t\sim\sqrt{n\log n}$ gives $2^{n-O(\sqrt{n\log n})}$ words $w$ satisfying~\eqref{e3}.
\end{proof}

{\bf Acknowledgement:} We would like to thank the anonymous referees for their helpful comments and their quick response.

\end{document}